\documentclass[12pt]{article}
\usepackage[active]{srcltx}
\usepackage{amssymb,latexsym,amsmath,enumerate,verbatim,amsfonts,amsthm}

\newtheorem{theorem}{Theorem}[section]
\newtheorem{corollary}[theorem]{Corollary}

\newtheorem{lemma}{Lemma}[section]
\newtheorem{remark}{Remark}[section]

\theoremstyle{definition}
\newtheorem{definition}{Definition}[section]
\numberwithin{equation}{section}
\begin{document}
    \title{Upper bounds for Z$_1$-eigenvalues of generalized Hilbert tensors
    \thanks{This  work was supported by
            the National Natural Science Foundation of P.R. China (Grant No.
            11571095, 11601134, 11701154).}}

    \date{}
\author{Juan Meng\thanks{School of Mathematics and Information Science,
        Henan Normal University, XinXiang HeNan,  P.R. China, 453007.
        Email: 1015791785@qq.com} , \quad    Yisheng Song\thanks{Corresponding author. School of Mathematics and Information Science  and Henan Engineering Laboratory for Big Data Statistical Analysis and Optimal Control,
        Henan Normal University, XinXiang HeNan,  P.R. China, 453007.
        Email: songyisheng@htu.cn.}}
\maketitle{}
   \begin{abstract}
        In this paper, we introduce the concept of Z$_1$-eigenvalue to infinite dimensional generalized Hilbert tensors (hypermatrix) $\mathcal{H}_\lambda^{\infty}=(\mathcal{H}_{i_{1}i_{2}\cdots i_{m}})$,
        $$
        \mathcal{H}_{i_{1}i_{2}\cdots i_{m}}=\frac{1}{i_{1}+i_{2}+\cdots i_{m}+\lambda},\
        \lambda\in \mathbb{R}\setminus\mathbb{Z}^-;\ i_{1},i_{2},\cdots,i_{m}=0,1,2,\cdots,n,\cdots,
        $$
        and proved that its $Z_1$-spectral radius is not larger than $\pi$ for $\lambda>\frac{1}{2}$, and is at most $\frac{\pi}{\sin{\lambda\pi}}$ for $\frac{1}{2}\geq \lambda>0$. Besides, the upper bound of $Z_1$-spectral radius of an $m$th-order $n$-dimensional generalized Hilbert tensor $\mathcal{H}_\lambda^n$ is obtained also, and such a bound only depends on $n$ and $\lambda$.\vspace{3mm}

        \noindent {\bf Key words:}\hspace{2mm}  Infinite-dimensional generalized Hilbert tensor, $Z_1$-eigenvalue, Spectral radius, Hilbert inqualities.

	\noindent {\bf AMS subject classifications (2010):}\hspace{2mm} 47H15, 47H12, 34B10, 47A52, 47J10, 47H09, 15A48, 47H07
    \vspace{3mm}

    \end{abstract}

\section{Introduction}

      A generalized Hilbert matrix has the form \cite{Hill}:
      \begin{equation}\label{eq11} H^\infty_\lambda=\left(\frac{1}{i+j+\lambda}\right)_{i, j\in \mathbb{Z}^+}\end{equation} where $\mathbb{Z}^+$ ($\mathbb{Z}^-$) is the set of all non-negative (non-positive) integers and $\lambda\in\mathbb{R}\setminus\mathbb{Z}^-.$ Denote such a Hilbert matrix with $i,j\in I_n=\{0,1,2,\cdots,n\}$ by $H_\lambda^n.$ When $\lambda=1$, such a matrix is called Hilbert matrix,  which was introduced by Hilbert \cite{H1894}. Choi \cite{C1983} and Ingham \cite{I1936} proved that Hilbert matrix $H_1^\infty$ is a bounded linear operator (but not compact operator) from Hilbert space $l^2$ into itself. Magnus \cite{M1950} and Kato \cite{K1957} studied the spectral properties of $H_1^\infty$.  Frazer \cite{F1946} and Taussky \cite{T1949} discussed some nice properties of $n$-dimensional Hilbert matrix $H_1^n$.
      Rosenblum \cite{R1958} showed that for a real $\lambda < 1$, $H^\infty_\lambda$ defines a bounded operator on $l^p$ for $2 < p < \infty$ and
that $\pi\sec\pi u$ is an eigenvalue of $H^\infty_\lambda$ for $|\Re u| < \frac12-\frac1p$.  For each non-integer complex number $\lambda$, Aleman, Montes-Rodr\'iguez,  Sarafoleanu \cite{AMS} showed that $H^\infty_\lambda$ defines a bounded linear operator on the Hardy spaces $H^p$ ($1 < p <\infty$).

      As a natural extension of a generalized Hilbert matrix, the generalized Hilbert tensor (hypermatrix) was  introduced by Mei and Song \cite{MS2017}. For each $\lambda\in\mathbb{R}\setminus\mathbb{Z}^-$, the entries of an $m$th-order infinite dimensional generalized Hilbert tensor $\mathcal{H}_\lambda^\infty=(\mathcal{H}_{i_{1}i_{2}\cdots
    i_{m}})$ are defined by \begin{equation}\label{eq12}
        \mathcal{H}_{i_{1}i_{2}\cdots i_{m}}=\frac{1}{i_{1}+i_{2}+\cdots i_{m}+\lambda},\
        i_{1},i_{2},\cdots,i_{m}=0,1,2,\cdots,n,\cdots.
        \end{equation} They showed $\mathcal{H}_\lambda^\infty$ defines a bounded and positively $(m-1)$-homogeneous operator from $l^{1}$
into $l^{p}\ (1<p<\infty)$. Song and Qi
\cite{SQ2014} studied the operator properties of Hilbert tensors $\mathcal{H}_1^\infty$ and the spectral properties of  $\mathcal{H}_1^n$.
 Such a tensor, $\mathcal{H}_\lambda^\infty$ may be refered to as a Hankel tensor with $v=(1,\frac12, \frac13,
\cdots,\frac{1}{n},\cdots)$. The concept of Hankel tensor was introduced by  Qi \cite{LQH}. For more further research of Hankel
tensors, see  Qi \cite{LQH}, Chen and Qi \cite{CQ2015}, Xu \cite{X2016}.  Denote such an $m$th-order $n$-dimensional generalized Hilbert tensor by $\mathcal{H}_\lambda^n.$ 

 For a real vector
$x=(x_{1},x_{2},\cdots,x_{n},x_{n+1},\cdots)\in l^1$,
$\mathcal{H}_\lambda^\infty x^{m-1}$ is an infinite dimensional vector with its $i$th component defined
by
    \begin{equation}\label{eq13}
    (\mathcal{H}_\lambda^\infty x^{m-1})_{i}
    =\sum_{i_{2},\cdots,i_{m}=0}^{\infty}\frac{x_{i_{2}}\cdots x_{i_{m}}}{i+i_{2}+\cdots +i_{m}+\lambda},\lambda\in \mathbb{R}\setminus\mathbb{Z}^-;\ i=0,1,2,\cdots.
    \end{equation}
Accordingly, $ \mathcal{H}_\lambda^\infty x^{m}$ is given by
    \begin{equation}\label{eq14}
    \mathcal{H}_\lambda^\infty x^{m}=\sum_{i_{1},i_{2},\cdots,i_{m}=0}^{\infty}\frac{x_{i_{1}}x_{i_{2}}\cdots x_{i_{m}}}{i_{1}+i_{2}+\cdots +i_{m}+\lambda},\lambda\in \mathbb{R}\setminus\mathbb{Z}^-.
\end{equation}
Mei and Song \cite{MS2017} proved that  $\mathcal{H}_\lambda^\infty x^{m}<\infty$ and $\mathcal{H}_\lambda^\infty x^{m-1}\in l^p\ (1<p<\infty)$ for all real vector
$x\in l^1$.

In this paper, we will  introduce the concept of $Z_1$-eigenvalue $\mu$ for an
$m$th-order infinite dimensional generalized Hilbert tensor
$\mathcal{H}_\lambda^\infty$ and will study  some upper bounds of $Z_1$-spectral radius for infinite dimensional generalized Hilbert tensor $\mathcal{H}_\lambda^\infty$ and $n$-dimensional generalized Hilbert tensor $\mathcal{H}_\lambda^n$.

In Section 2, we will give some Lemmas and basic conclusions, and introduce the concept of $Z_1$-eigenvalue. 
In Section 3, with the help of the Hilbert type inequalities, the upper bound of
$Z_1$--spectral radius of $\mathcal{H}_\lambda^\infty$ with $\lambda>0$ is at
most $\pi$ when $\lambda>\frac{1}{2}$, and is not larger than
$\frac{\pi}{\sin{\lambda\pi}}$ when $0<\lambda\leq\frac{1}{2}$. Furthermore, for each $Z_1$-eigenvalue $\mu$ of $\mathcal{H}_\lambda^n$,  $|\mu|$ is smaller than or equal to $C(n,\lambda)$, where $C(n,\lambda)$ only depends on the structured coefficient $\lambda$ of generalized Hilbert tensor and the dimensionality $n$ of  European space.

    \section{Preliminaries and Basic Results}

    For $0<p<\infty$, $l^p$ is a space consisting of all real number sequences $x=(x_i)_{i=1}^{+\infty}$ satisfying
$\sum\limits_{i=1}^{+\infty}|x_{i}|^{p}<\infty.$  If $p\geq1$, then a
norm on $l^p$ is defined by
$$\|x\|_{l^p}=\left(\sum_{i=1}^{+\infty}|x_{i}|^p\right)^{\frac{1}{p}}.$$
It is well known that $l^2$ is a Hilbert space with the inner product $$\langle x, y\rangle=\sum_{i=0}^{+\infty}x_iy_i.$$ Clearly, $\|x\|_{l^2}=\sqrt{\langle x,x\rangle}.$

For $p\geq1$, a norm $\mathbb{R}^n$ can be defined by
	$$\|x\|_{p}=\left(\sum_{i=1}^{n}|x_{i}|^{p}\right)^{\frac{1}{p}}.$$
It is well known that
	\begin{equation}\label{eq21}
	\|x\|_2\leq \|x\|_1\leq \sqrt{n}\|x\|_2.
	\end{equation}

The following Hilbert type inequalities  were proved by Frazer \cite{F1946} on $\mathbb{R}^n$ and Ingham \cite{I1936} on $l^2$, respectively.
\begin{lemma}(Frazer \cite{F1946})\label{lem21}
Let $x=(x_1,x_2,\cdots,x_n)^{\top}\in \mathbb{R}^n$. Then
\begin{equation}\label{eq22}
\sum\limits_{i=0}^{n}\sum\limits_{j=0}^{n}\frac{|x_i||x_j|}{i+j+1}\leq
(n \sin\frac{\pi}{n})\sum\limits_{k=0}^{n}x_{k}^2=\|x\|_2^{2}n \sin\frac{\pi}{n},
\end{equation}
\end{lemma}

\begin{lemma}(Ingham \cite{I1936})\label{lem22}
Let $x=(x_1,x_2,\cdots,x_n,\cdots)^{\top}\in l^2$ and $a>0$. Then
\begin{equation}\label{eq23}
\sum\limits_{i=0}^{\infty}\sum\limits_{j=0}^{\infty}\frac{|x_i||x_j|}{i+j+a}\leq
M(a)\sum\limits_{k=0}^{\infty}x_{k}^2=M(a)\|x\|_{l^2}^2,
\end{equation}
where
$$M(a)=\begin{cases}\frac{\pi}{\sin{a\pi}},&0<a\leq\frac{1}{2};\\
                            \pi,&a>\frac{1}{2}.
\end{cases}$$

\end{lemma}

An $m$-order $n$-dimensional tensor (hypermatrix) $\mathcal{A} = (a_{i_1\cdots i_m})$ is a multi-array of real entries $a_{i_1\cdots i_m}\in\mathbb{R}$, where $i_j \in I_n=\{1,2,\cdots,n\}$ for $j \in [m]=\{1,2,\cdots,m\}$. We use $T_{m,n}$ denote the set of all real $m$th-order $n$-dimensional tensors. Then $\mathcal{A}\in T_{m,n}$ is
called a  symmetric tensor if the entries $a_{i_1\cdots
i_m}$are invariant under any permutation of their indices.
$\mathcal{A}\in T_{m,n}$ is called  nonnegative (positive) if
$a_{i_{1}i_{2}\cdots i_{m}}\geq 0 (a_{i_{1}i_{2}\cdots i_{m}}> 0)$
for all $i_{1},i_{2},\cdots,i_{m}$.

\begin{definition}(Chang and Zhang \cite{CZ})\label{def21} Let $\mathcal{A}\in T_{m,n}.$
A number $\mu\in\mathbb{R}$ is called $Z_1$-eigenvalue of
$\mathcal{A}$ if there is a real vector $x$ such that
\begin{equation}\label{eq24}
\begin{cases}\mathcal{A}x^{m-1}=\mu x\\
\|x\|_1=1
\end{cases}
\end{equation}
and call such a vector $x$ an $Z_1$-eigenvector associated with $\mu$.
\end{definition}

 For the concepts of eigenvalues of higher order tensors, Qi \cite{LQ1,LQ2} first used and introduced them for symmetric tensors, and Lim \cite{LL} independently introduced this notion but restricted $x$ to be a
real vector and $\lambda$ to be a real number. Subsequently,  the spectral
properties of nonnegative matrices had been generalized to
$n$-dimensional nonnegative tensors under various conditions by
Chang et al. \cite{CPZ13,CPT1}, He and Huang \cite{HH2014}, He
\cite{H2016}, He et al. \cite{HLK2016}, Li et al. \cite{WLV2015}, Qi
\cite{LQ13}, Song and Qi \cite{SQ13,SQ14}, Wang et al.
\cite{WZS2016}, Yang and Yang \cite{YY11,YY10} and references therein. The notion of Z$_1$-eigenvalue was introduced by Chang and Zhang \cite{CZ} for higher Markov chains. Now we introduce it to infinite dimensional generalized Hilbert tensors.

Let \begin{equation}\label{eq25}
	T_{\infty}x =\begin{cases}
	\|x\|_{l^1}^{2-m}\mathcal{H}^{\infty}_\lambda x^{m-1},    & x\neq \theta \\
	\theta,  &  x=\theta,
	\end{cases}
	\end{equation}
	where $\theta=(0,0,\cdots,0,\cdots)$. Mei and Song
\cite{MS2017} first used the concept of the operator $T_\infty$ induced by a generalized Hilbert tensor $\mathcal{H}_\lambda^\infty$  and showed $T_\infty$ is a bounded and positively homogeneous operator from $l^{1}$ into $l^{p}\ (1<p<\infty)$. Then $T_\infty$ is refered to as a bounded and positively homogeneous operator from $l^2$ into $l^2$. So, the concept of Z$_1$-eigenvalue may be introduced to the infinite dimensional Hilbert tensor $\mathcal{H}_\lambda^\infty$.

\begin{definition}\label{def22}
Let $\mathcal{H}_\lambda^\infty$ be an $m$th-order infinite dimensional generalized Hilbert tensor. A real number $\mu$ is called  a $Z_1$-eigenvalue of $\mathcal{H}_\lambda^\infty$ if there exists a nonzero vector $x\in l^2$  satisfying
\begin{equation}\label{eq26} T_{\infty}x=\|x\|_{l^1}^{2-m}\mathcal{H}^{\infty}_\lambda x^{m-1}=\mu x.\end{equation}
Such a vector $x$ is called an $Z_1$-eigenvector associated with $\mu$.
\end{definition}

\section{Main Results}

\begin{theorem}\label{thm31}
Let $\mathcal{H}_\lambda^{n}$ be an $m$th-order $n$-dimensional
generalized Hilbert tensor. Then
$$|\mu|\leq C(n,\lambda) \mbox{ for all $Z_1$-eigenvalue  $\mu$ of $\mathcal{H}_\lambda^{n}$},$$
where $[\lambda]$ is the largest integer  not exceeding $\lambda$ and $$C(n,\lambda)=\begin{cases}n \sin\frac{\pi}{n},&\lambda\geq1;\\
			\frac{n}\lambda,&1>\lambda>0;\\
			\frac{n}{\min\{\lambda-[\lambda],1+[\lambda]-\lambda\}},&-mn<\lambda<0;\\
\frac{n}{-mn-a},&\lambda<-mn.
		\end{cases}$$
\end{theorem}
\begin{proof} For $\lambda\geq1$,
 it follows from Lemma \ref{lem21} that for all nonzero vector $x\in \mathbb{R}^n$,
$$
\aligned
|\mathcal{H}_\lambda^{n}x^m|&=\left|\sum\limits_{i_1,i_2,\cdots,i_{m}=0}^{n}\frac{x_{i_1}x_{i_2}\cdots
x_{i_m}}{i_{1}+i_{2}+\cdots+i_{m}+\lambda}\right|\\
&\leq\sum\limits_{i_1,\cdots,i_m =0}^n
\frac{|x_{i_1}x_{i_2}\cdots
x_{i_m}|}{i_1+i_2+\underbrace{0+\cdots+0}\limits_{m-2}+\lambda}\\
&=\sum\limits_{i_1,i_2,\cdots,i_{m}=0}^{n}\frac{|x_{i_1}||x_{i_2}|\cdots
|x_{i_m}|}{i_{1}+i_{2}+\lambda}\\
&=\left(\sum\limits_{i_1=0}^{n}
\sum\limits_{i_2=0}^{n}\frac{|x_{i_1}||x_{i_2}|}{i_{1}+i_{2}+\lambda}\right)\sum\limits_{i_3,i_4,\cdots,i_{m}=0}^{n}|x_{i_3}||x_{i_4}|\cdots
|x_{i_m}|\\
&\leq\left(\sum\limits_{i_1=0}^{n}
\sum\limits_{i_2=0}^{n}\frac{|x_{i_1}||x_{i_2}|}{i_{1}+i_{2}+1}\right)\sum\limits_{i_3,i_4,\cdots,i_{m}=0}^{n}|x_{i_3}||x_{i_4}|\cdots
|x_{i_m}|\\
&\leq(\|x\|_2^{2}n \sin\frac{\pi}{n})\left(\sum\limits_{i=0}^{n}|x_i|\right)^{m-2}\\
&=\|x\|_2^{2}\|x\|_1^{m-2}n \sin\frac{\pi}{n}.
\endaligned
$$
That is, \begin{equation}\label{eq31} |\mathcal{H}_\lambda^{n}x^m|\leq \|x\|_2^{2}\|x\|_1^{m-2}n \sin\frac{\pi}{n}.\end{equation}
Since $\mu$  is a $Z_1$-eigenvalue of $\mathcal{H}_\lambda^{n}$, then there exists a nonzero vector $x$ such that
\begin{equation}\label{eq32}\mathcal{H}_\lambda^{n}x^{m-1}=\mu x\mbox{ and }\|x\|_1=1.\end{equation}
Thus, we have,
$$|\mu x^\top x|=|x^\top (\mathcal{H}_\lambda^{n}x^{m-1})| =|\mathcal{H}_\lambda^{n}x^{m}|\leq \|x\|_2^{2}\|x\|_1^{m-2}n \sin\frac{\pi}{n},$$
 and then, $$|\mu|\|x\|_2^2\leq \|x\|_2^{2}\|x\|_1^{m-2}n \sin\frac{\pi}{n}.$$ As a result, \begin{equation}\label{eq33}|\mu|\leq n\sin\frac{\pi}{n}.\end{equation}

For all $\lambda\in \mathbb{R}\setminus\mathbb{Z}^-$ with $\lambda<1$,  it is obvious that for $1>\lambda>0$,
		$$
	\min_{i_{1},\cdots,i_{m}\in I_n}	|i_{1}+i_{2}+\cdots +i_{m}+\lambda|=\lambda.
		$$
	For $-mn<\lambda<0$, there exist some positive integers $i'_{1},i'_{2},\cdots,i'_{m}$ and $i''_{1},i''_{2},\cdots,i''_{m}$ such that $$i'_{1}+i'_{2}+\cdots+i'_{m}=-[\lambda] \mbox{ and } i''_{1}+i''_{2}+\cdots+i''_{m}=-[\lambda]-1,$$
		and hence, $$
	\min_{i_{1},\cdots,i_{m}\in I_n}	|i_{1}+i_{2}+\cdots +i_{m}+\lambda|=\min\{\lambda-[\lambda],\lambda-(-[\lambda]-1)\}.
		$$
For $\lambda<-mn$, we also have,
$$
	\min_{i_{1},\cdots,i_{m}\in I_n}	|i_{1}+i_{2}+\cdots +i_{m}+\lambda|=|mn+\lambda|=-mn-\lambda.
		$$
Therefore,  we have for $ \lambda\in \mathbb{R}\setminus\mathbb{Z}^- $ with $\lambda<1$,
		$$
		\frac1{|i_{1}+i_{2}+\cdots +i_{m}+\lambda|}\leq N(\lambda)=\begin{cases}
			\frac1\lambda,&1>\lambda>0;\\
			\frac1{\min\{\lambda-[\lambda],1+[\lambda]-\lambda\}},&-mn<\lambda<0;\\
\frac1{-mn-a},&\lambda<-mn.
		\end{cases}$$
Then, for all nonzero vector $x\in \mathbb{R}^n$, we have
$$
\aligned
|\mathcal{H}_\lambda^{n}x^m|&=\left|\sum\limits_{i_1,i_2,\cdots,i_{m}=0}^{n}\frac{x_{i_1}x_{i_2}\cdots
x_{i_m}}{i_{1}+i_{2}+\cdots+i_{m}+\lambda}\right|\\
&\leq\sum\limits_{i_1,\cdots,i_m =0}^n
\frac{|x_{i_1}x_{i_2}\cdots
x_{i_m}|}{|i_1+i_2+\cdots+i_{m}+\lambda|}\\
&\leq N(\lambda)\sum\limits_{i_1,i_2,\cdots,i_{m}=0}^{n}|x_{i_1}||x_{i_2}|\cdots
|x_{i_m}|\\
&= N(\lambda)\left(\sum\limits_{i=0}^{n}|x_i|\right)^m=N(\lambda)\|x\|_1^{m}.
\endaligned
$$
For each Z$_1$-eigenvalue $\mu$ of $\mathcal{H}_\lambda^{n}$ with its eigenvector $x$, from \eqref{eq32} and $\|x\|_1\leq \sqrt{n}\|x\|_2$, it follows taht
$$|\mu|(\frac1n\|x\|_1^2)\leq|\mu|\|x\|_2^2=|\mathcal{H}_\lambda^{n}x^m|\leq N(\lambda)\|x\|_1^{m},$$
and hence, $$|\mu|\leq nN(\lambda).$$
This completes the proof.
\end{proof}

When $\lambda=1$, the following conclusion of Hilbert tensor is easily obtained. Also see Song and Qi \cite{SQ2014} for the conclusions about H-eigenvalue and Z-eigenvalue of such a tensor.
\begin{corollary} \label{cor31}
Let $\mathcal{H}$ be an $m$th-order $n$-dimensional
Hilbert tensor. Then
for all $Z_1$-eigenvalue $\mu$ of $\mathcal{H}$,
$$|\mu|\leq n \sin\frac{\pi}{n}.$$
\end{corollary}

\begin{theorem}\label{thm32}
Let $\mathcal{H}_\lambda^{\infty}$ be an $m$th-order infinite dimensional
generalized Hilbert tensor. Assume $\lambda>0$, then
for $Z_1$-eigenvalue $\mu$ of $\mathcal{H}_\lambda^{\infty}$,
$$|\mu|\leq M(\lambda)=\begin{cases}\frac{\pi}{\sin{\lambda\pi}},&0<\lambda\leq\frac{1}{2};\\
                                            \pi,&\lambda>\frac{1}{2}.\end{cases}$$
\end{theorem}
\begin{proof}For $x\in l^2$, it follows from Lemma \ref{lem22} that
$$
\aligned
|\langle x,\mathcal{H}_\lambda^{\infty}x^{m-1}\rangle|=|\mathcal{H}_\lambda^{\infty}x^m|&=\left|\sum\limits_{i_1,i_2,\cdots,i_{m}=0}^{+\infty}\frac{x_{i_1}x_{i_2}\cdots
x_{i_m}}{i_{1}+i_{2}+\cdots+i_{m}+\lambda}\right|\\
&\leq\sum\limits_{i_1,\cdots,i_m =0}^{+\infty}
\frac{|x_{i_1}x_{i_2}\cdots
x_{i_m}|}{i_1+i_2+\underbrace{0+\cdots+0}\limits_{m-2}+\lambda}\\
&=\sum\limits_{i_1,i_2,\cdots,i_{m}=0}^{+\infty}\frac{|x_{i_1}||x_{i_2}|\cdots
|x_{i_m}|}{i_{1}+i_{2}+\lambda}\\
&=\left(\sum\limits_{i_1=0}^{+\infty}
\sum\limits_{i_2=0}^{+\infty}\frac{|x_{i_1}||x_{i_2}|}{i_{1}+i_{2}+\lambda}\right)\sum\limits_{i_3,i_4,\cdots,i_{m}=0}^{+\infty}|x_{i_3}||x_{i_4}|\cdots
|x_{i_m}|\\
&=\left(\sum\limits_{i_1=0}^{+\infty}
\sum\limits_{i_2=0}^{\infty}\frac{|x_{i_1}||x_{i_2}|}{i_{1}+i_{2}+\lambda}\right)\left(\sum\limits_{i=0}^{+\infty}|x_i|\right)^{m-2}\\
&\leq M(\lambda) \|x\|_{l^2}^2\|x\|_{l^1}^{m-2},
\endaligned$$
and so, \begin{equation}\label{eq34} |\langle x,T_\infty x\rangle|=|\langle x,\|x\|_1^{2-m}\mathcal{H}_\lambda^{\infty}x^{m-1}\rangle|=\|x\|_{l^1}^{2-m}|\mathcal{H}_\lambda^{\infty}x^{m}|\leq M(\lambda) \|x\|_{l^2}^2.\end{equation}

For each $Z_1$-eigenvalue $\mu$ of $\mathcal{H}_\lambda^{\infty}$, there exists a nonzero vector $x\in l^2$ such that
$$T_{\infty}x=\|x\|_{l^1}^{2-m}\mathcal{H}_\lambda^{\infty}x^{m-1}=\mu x,$$
and so, $$\mu\|x\|_{l^2}^2=\mu \langle x, x\rangle=\langle x,\|x\|_{l^1}^{2-m}\mathcal{H}_\lambda^{\infty}x^{m-1}\rangle=\|x\|_{l^1}^{2-m}\mathcal{H}_\lambda^{\infty}x^{m}.$$
Therefore, we have
$$|\mu|\|x\|_{l^2}^2=\|x\|_{l^1}^{2-m}|\mathcal{H}_\lambda^{\infty}x^{m}|\leq M(\lambda)\|x\|_{l^2}^2,$$
and then, $$|\mu|\leq M(\lambda).$$
This completes the proof.
\end{proof}

When $\lambda=1$, the following conclusion of infinite dimensional Hilbert tensor is easily obtained.
\begin{corollary} \label{cor32}
Let $\mathcal{H}_{\infty}$ be an $m$th-order infinite dimensional
Hilbert tensor. Then
for all $Z_1$-eigenvalue $\mu$ of $\mathcal{H}_{\infty}$,
$$|\mu|\leq\pi.$$
\end{corollary}
\begin{remark}
\begin{itemize}
\item[(i)] In Theorem \ref{thm31}, the upper bound of $Z_1$-eigenvalue of $\mathcal{H}_\lambda^{n}$
are showed. However the upper bound may not be the best. Then which number is its best upper bounds?
\item[(ii)] In Theorem \ref{thm32}, the upper bound of $Z_1$-eigenvalue of $\mathcal{H}_\lambda^{\infty}$
are showed for $\lambda>0$, then for $\lambda<0$ with $\lambda\in
\mathbb{R}\setminus\mathbb{Z}^-$, it is unknown whether have similar
conclusions or not. And it is not clear whether the upper
bound may be attained or cannot be attained.
\end{itemize}
\end{remark}

\end{document}